\documentclass[12pt]{amsart}   
\usepackage{amsmath,amsthm,amssymb}
\newdimen\slantmathcorr
\def\oversl#1{
\setbox0=\hbox{$#1$}
\slantmathcorr=\wd0
\hskip 0.2\slantmathcorr \overline{\hbox to 0.8\wd0{
\vphantom{\hbox{$#1$}}}}
\hskip-\wd0\hbox{$#1$}
}

\newtheorem{theorem}{Theorem}[section]
\newtheorem{Lemma}[theorem]{Lemma} 
\newtheorem{Proposition}[theorem]{Proposition}
\newtheorem{Corollary}[theorem]{Corollary}
\newtheorem{Definition}[theorem]{Definition} 
\newtheorem{Example}[theorem]{Example}
 
\begin{document}
\title{MEAN ERGODIC SHADOWING}           
\author{Pramod Kumar Das$^1$, Tarun Das$^2$}         
\begin{abstract}
We introduce and study a new variant of shadowing namely mean ergodic shadowing. We establish relationship of this variant with several other variants of shadowing. We show that a minimal system with shadowing cannot have mean ergodic shadowing. We give a necessary and sufficient condition for an orbital limit function to have mean ergodic shadowing property.            
\end{abstract}  
\maketitle 



\section{Introduction} 

Throughout this paper, a dynamical system is a pair $(X,f)$, where $X$ is a compact metric space (To avoid technical difficulties, we assume that $X$ contains at least two points) and $f:X\rightarrow X$ is a continuous map. The concept of shadowing was originated from Anosov closing lemma and because of its rich consequences, shadowing plays an important role in topological dynamics. It is believed that the study of certain variants or generalizations of a particular interesting concept give new insights. This belief gave the impetus to study several variants of shadowing including ergodic shadowing \cite{FG}, sub-shadowing \cite{DH} and average shadowing \cite{WOC}. These variants share a common motivation of studying the behaviour of a dynamical system by exploiting the closeness of approximate orbits and exact orbits. The similarities and differences between shadowing and each of these variants help to understand better the behaviour of a system with shadowing. Another important concept in topological dynamics is minimality. In a minimal system, the orbit of individual point is dense in the phase space. In \cite{FD,S}, authors investigated relationship between the dynamics of a sequence of functions and its limit function. In \cite{KDD}, authors gave a necessary and sufficient condition for a point to be a shadowable point (introduced in \cite{M}) of an orbital limit function.      

Here, we first introduce and study a new variant of shadowing namely mean ergodic shadowing. We establish relationships of this new variant with several other variants of shadowing. We show that a minimal system with shadowing cannot possess mean ergodic shadowing. Finally, we give a necessary and sufficient condition for an orbital limit function to possess mean ergodic shadowing.     
\par\noindent\rule{\textwidth}{0.4pt}
{\footnotesize Pramod Kumar Das, \textit{Email}: pramodkumar.das@nmims.edu\\
Tarun Das, \textit{Email}: tarukd@gmail.com\\
$^1$School of Mathematical Sciences, Narsee Monjee Institute of Management Studies, Vile Parle (West), Mumbai-400056.\\
$^2$Department of Mathematics, Faculty of Mathematical Sciences, University of Delhi, New Delhi-110007.}    

\section{Preliminaries} 

Let $(X,d)$ be a compact metric space and let $f$ be a continuous map on $X$. Let $\mathbb{N}$ be the set of all non-negative integers. The density and upper density of $E\subset \mathbb{N}$ is defined respectively by $d(E)=$lim$_{n\to\infty}\frac{\#(E\cap [0,n-1])}{n}$ and $\overline{d}(E)=$lim sup$_{n\to\infty}\frac{\#(E\cap [0,n-1])}{n}$, where $\#(B)$ denotes the cardinality of $B\subset \mathbb{N}$ and $[0,n-1]=\lbrace 0,1,2,3,...,n-1\rbrace$.       
\medskip

Let $\xi=\lbrace x_i\rbrace_{i\in\mathbb{N}}$ be a sequence in $X$ and $\epsilon$, $\delta$ be positive real numbers. Then,  
\medskip

(i) $\xi$ is said to be a $\delta$-pseudo orbit if $d(f(x_i),x_{i+1})<\delta$ for all $i\in\mathbb{N}$ and $\xi$ is said to be $\epsilon$-shadowed by some $x\in X$ if $d(f^i(x),x_i)<\epsilon$ for all $i\in\mathbb{N}$. $f$ is said to have shadowing if for every $\epsilon>0$, there is $\delta>0$ such that every $\delta$-pseudo orbit is $\epsilon$-shadowed by some point in $X$.    
\medskip

(ii) $\xi$ is said to be a $\delta$-ergodic pseudo orbit if $d(f(x_i),x_{i+1})<\delta$ for all $i\in\mathbb{N}$ except a set of density zero. $\xi$ is said to be $\epsilon$-ergodic shadowed by some $x\in X$ if $d(f^i(x),x_i)<\epsilon$ for all $i\in\mathbb{N}$ except a set of density zero. $\xi$ is said to be $\epsilon$-ergodic shadowed with index $\alpha$ by some $x\in X$ if $\overline{d}(B(x,\xi,\epsilon))>\alpha$, where $B(x,\xi,\epsilon)=\lbrace i\in\mathbb{N}\mid d(f^i(x),x_i)< \epsilon\rbrace$. $f$ is said to have ergodic shadowing \cite{FG} if for every $\epsilon>0$, there is $\delta>0$ such that every $\delta$-ergodic pseudo orbit is $\epsilon$-ergodic shadowed by some point in $X$. $f$ is said to have $M^{\alpha}$-shadowing \cite{WOC} if for every $\epsilon>0$, there is $\delta>0$ such that every $\delta$-ergodic pseudo orbit is $\epsilon$-ergodic shadowed with index $\alpha$, by some point in $X$. So, $\overline{d}$-shadowing \cite{DH} is a special case for $\alpha=\frac{1}{2}$. $f$ is said to have $\underline{d}$-shadowing \cite{DH} if for every $\epsilon>0$, there is $\delta>0$ such that for every $\delta$-ergodic pseudo orbit $\xi=\lbrace x_i\rbrace_{i\in\mathbb{N}}$ there exists $x\in X$ such that $\underline{d}(B(x,\xi,\epsilon))>0$.        
\medskip

(iii) $ \xi$ is said to be a $\delta$-average pseudo orbit if there is $N\in\mathbb{N}$ such that for each $n\geq N$ and any non-negative integer $k$, lim sup$_{n\to\infty} \frac{1}{n}\Sigma_{i=0}^{n-1}d(f(x_{i+k}),x_{i+k+1})<\delta$. $\xi$ is said to be an asymptotic average pseudo orbit if lim$_{n\to\infty}\frac{1}{n}\Sigma_{i=0}^{n-1}d(f(x_i),x_{i+1})=0$. $\xi$ is said to be a $\delta$-asymptotic average pseudo orbit if lim sup$_{n\to \infty} \Sigma_{i=0}^{n-1} d(f(x_i),x_{i+1})<\delta$. $\xi$ is said to be $\epsilon$-shadowed in average by some $x\in X$ if lim sup$_{n\to\infty}\frac{1}{n}\Sigma_{i=0}^{n-1}d(f^i(x),x_i)<\epsilon$. $\xi$ is said to be asymptotically shadowed in average by some $x\in X$ if lim$_{n\to\infty}\frac{1}{n}\Sigma_{i=0}^{n-1}d(f^i(x),x_i)=0$. $f$ is said to have average shadowing if for every $\epsilon>0$, there is $\delta>0$ such that every $\delta$-average pseudo orbit is $\epsilon$-shadowed in average by some point in $X$. $f$ is said to have asymptotic average shadowing if every asymptotic average pseudo orbit is asymptotically shadowed in average by some point in $X$. $f$ is said to have weak asymptotic average shadowing if for every $\epsilon>0$, any asymptotic average pseudo orbit is $\epsilon$-shadowed in average by some point in $X$.    
\medskip

We define $N(U,V)=\lbrace n\in\mathbb{N}^+\mid f^n(U)\cap V \neq\phi\rbrace$ for two non-empty open sets $U,V\subset X$ and $N(x,U)=\lbrace n\in\mathbb{N}^+ \mid f^n(x)\in U\rbrace$ for non-empty open set $U\subset X$ containing $x$. $f$ is called transitive if $N(U,V)\neq\phi$ and totally transitive if $f^k$ is transitive for each $k\in \mathbb{N}^+$. A point $x\in X$ is called syndetically recurrent if $N(x,U)$ is syndetic for every neighborhood $U$ of $x$. If $M(g)$ denote the set of all syndetically recurrent points of $g$, then note that $M(f)=M(f^k)$ for each $k\in \mathbb{N}^+$. A closed set $L$ is minimal for $f$ if $L=\overline{O_f(x)}$ for some $x\in M(f)$. If the phase space $X$ is minimal, then we say that $f$ is minimal. If $f$ is minimal, then all points of the phase space are syndetically recurrent. But all points syndetically recurrent does not imply that the system is minimal. A finite $\delta$-chain is a finite set of points $\lbrace x_0,x_1,x_2,x_3,...,x_n\rbrace$ such that $d(f(x_i),x_{i+1})<\delta$ for all $0\leq i\leq (n-1)$. We say that $f$ is chain transitive if for any $\epsilon>0$ and any pair of points $x,y\in X$ there is a finite $\epsilon$-chain from $x$ to $y$ (i.e. $x_0=x$ and $x_n=y$) and that $f$ is totally chain transitive if $f^k$ is chain transitive for each $k\in \mathbb{N}^+$.      

\section{Mean Ergodic Shadowing and Other Variants}       

In this section, we introduce a new variant of shadowing called mean ergodic shadowing and study its properties and its relationships with other variants of shadowing.     

\begin{Definition}
Let $f:X\rightarrow X$ be a continuous map on a compact metric space $X$. Then $f$ is said to have mean ergodic shadowing if for every $\epsilon>0$, there is $\delta>0$ such that every $\delta$-ergodic pseudo orbit can be $\epsilon$-shadowed in average by some point in $X$. 
\end{Definition}  

\begin{Lemma} 
Let $f:X\rightarrow X$ be a continuous map on a compact metric space $X$. Then, $f$ has mean ergodic shadowing if and only if for every $\epsilon>0$, there is $\delta>0$ such that every $\delta$-ergodic pseudo orbit is $\epsilon$-shadowed except on a set of upper density $\epsilon$.       
\label{2.1.1} 
\end{Lemma} 

\begin{proof}
Suppose that $f$ has mean ergodic shadowing. For every $\epsilon>0$ there is $\delta>0$ such that if $\xi=\lbrace x_i\rbrace_{i\in\mathbb{N}}$ is a $\delta$-ergodic pseudo orbit, then there is $x\in X$ such that 

lim sup$_{n\to\infty}\frac{1}{n}\Sigma_{i=0}^{n-1}d(f^i(x),x_i)<\epsilon^2$. 

If $E=\lbrace i\in\mathbb{N}\mid d(f^i(x),x_i)\geq \epsilon\rbrace$, then 

$\epsilon^2>$lim sup$_{n\to\infty}\frac{1}{n}\Sigma_{i=0}^{n-1}d(f^i(x),x_i)\geq$ lim sup$_{n\to\infty}\frac{1}{n}(\epsilon\#([0,n-1]\cap E))=\epsilon\overline{d}(E)$. 

This implies that $\overline{d}(E)<\epsilon$.    
\medskip
 
Conversely, fix $\epsilon>0$ and choose $\eta<\frac{\epsilon}{diam(X)+1}$. Let $\delta>0$ be such that if $\xi=\lbrace x_i\rbrace_{i\in\mathbb{N}}$ is a $\delta$-ergodic pseudo orbit, then $\xi$ is $\eta$-shadowed except on a set of upper density less than $\epsilon$, by some point $x$ in $X$. Thus, if $E=\lbrace i\in\mathbb{N}\mid d(f^i(x),x_i)\geq \eta\rbrace$, then $\overline{d}(E)<\eta$ and 

lim sup$_{n\to\infty}\frac{1}{n}\Sigma_{i=0}^{n-1} d(f^i(x),x_i)$

$\leq $lim sup$_{n\to\infty}\frac{1}{n}(diam(X)\#([0,n-1]\cap E)+\eta n)$

$\leq diam(X)\overline{d}(E)+\eta<\epsilon$.

This implies that $f$ has mean ergodic shadowing.
\end{proof} 

In view of the above lemma, the following definition of mean ergodic shadowing is useful in our study.   

\begin{Definition}
Let $f:X\rightarrow X$ be a continuous map on a compact metric space $X$. Then $f$ has mean ergodic shadowing if for every $\epsilon>0$, there is $\delta>0$ such that if $\xi=\lbrace x_i\rbrace_{i\in\mathbb{N}}$ is a $\delta$-ergodic pseudo orbit, there is $x\in X$ such that $\overline{d}(B^c(x,\xi,\epsilon))<\epsilon$, where $B^c(x,\xi,\epsilon)=\lbrace i\in \mathbb{N}\mid d(f^i(x),x_i)\geq\epsilon\rbrace$.    
\label{2.1.3}  
\end{Definition}

It follows from Definition \ref{2.1.3} that ergodic shadowing implies mean ergodic shadowing. Further, it is immediate from the identity $\underline{d}(A)=1-\overline{d}(A^c)$ that mean ergodic shadowing implies $\underline{d}$-shadowing.      

\begin{theorem} 
Let $f$ and $g$ be two continuous maps on $(X,d_1)$ and $(Y,d_2)$, respectively. If $h:X\rightarrow Y$ is a homeomorphism, then $g=h\circ f\circ h^{-1}$ has mean ergodic shadowing if and only if $f$ has mean ergodic shadowing. 
\end{theorem} 

\begin{proof}
Suppose that $f$ has mean ergodic shadowing. Let $\epsilon>0$ and $0<\epsilon'<\epsilon$ be given for $\epsilon$ by uniform continuity of $h^{-1}$, i.e. $d_1(x,y)<\epsilon'$ implies $d_2(h(x),h(y))<\epsilon$. Further, let $\delta'>0$ be such that every $\delta'$-ergodic pseudo orbit of $f$ is $\epsilon'$-shadowed except on a set of upper density less than $\epsilon'$ and let $\delta>0$ be given for $\delta'$ by uniform continuity of $h^{-1}$, i.e. $d_2(x,y)<\delta$ implies $d_1(h^{-1}(x),h^{-1}(y))<\delta'$.   
\medskip

Let $\xi=\lbrace x_i\rbrace_{i\in\mathbb{N}}$ be a $\delta$-ergodic pseudo orbit for $g=h\circ f\circ h^{-1}$, i.e. $d(E)=0$, where 

$E=\lbrace i\in\mathbb{N}\mid d_2((h\circ f\circ h^{-1})(x_i),x_{i+1})\geq\delta\rbrace\supset \lbrace i\in\mathbb{N}\mid d_1(f(h^{-1}(x_i)),h^{-1}(x_{i+1}))\geq\delta'\rbrace$.
\medskip

This shows that $\lbrace h^{-1}(x_i)\rbrace_{i\in\mathbb{N}}$ is a $\delta'$-ergodic pseudo orbit for $f$. So by the mean ergodic shadowing of $f$, there is $x\in X$ such that $\overline{d}(E')<\epsilon'$, where 

$E'=\lbrace i\in\mathbb{N}\mid d_1(f^i(x),h^{-1}(x_i))\geq\epsilon'\rbrace$

$\supset\lbrace i\in\mathbb{N}\mid d_2(h(f^i(x)),x_i)\geq\epsilon\rbrace$ 

$=\lbrace i\in\mathbb{N}\mid d_2((h\circ f^i\circ h^{-1})(h(x)),x_i)\geq\epsilon\rbrace$

$=\lbrace i\in\mathbb{N}\mid d_2(g^i(h(x)),x_i)\geq\epsilon\rbrace$.
\medskip

If $E''=\lbrace i\in\mathbb{N}\mid d_2(g^i(h(x)),x_i)\geq\epsilon\rbrace$, we have $\overline{d}(E'')<\epsilon$ because $\overline{d}(E')<\epsilon'<\epsilon$. Thus, $g$ has mean ergodic shadowing. The converse holds because $h$ is a homeomorphism.    
\end{proof} 

\begin{theorem}
Let $f$ and $g$ be two continuous maps on $(X,d_1)$ and $(Y,d_2)$, respectively. If $f$ and $g$ has mean ergodic shadowing, then $f\times g:X\times Y\rightarrow X\times Y$ has mean ergodic shadowing.    
\end{theorem} 

\begin{proof}
Here, $(X\times Y, d)$ is a compact metric space, where $d:(X\times Y)\times (X\times Y) \rightarrow \mathbb{R}$ be given by $d(x_1\times y_1, x_2\times y_2)=$ max $\lbrace d_1(x_1,x_2),d_2(y_1,y_2)\rbrace$. Let us fix $\epsilon>0$. Let $\delta_1$ and $\delta_2>0$ be given for $\frac{\epsilon}{2}$ by mean ergodic shadowing of $f$ and $g$, respectively. Let $\delta=$ min $\lbrace \delta_1,\delta_2\rbrace$ and $\lbrace (x_i\times y_i)_{i\in\mathbb{N}}\rbrace$ be a $\delta$-ergodic pseudo orbit for $f\times g$. Then $\lbrace x_i\rbrace_{i\in\mathbb{N}}$ and $\lbrace y_i\rbrace_{i\in\mathbb{N}}$ are $\delta_1$ and $\delta_2$ ergodic pseudo orbit for $f$ and $g$, respectively. So, there is $x\in X$ and $y\in Y$ such that $\overline{d}(A)<\frac{\epsilon}{2}$ and $\overline{d}(B)<\frac{\epsilon}{2}$, where $A=\lbrace i\in\mathbb{N}\mid d_1(f^i(x),x_i)\geq\frac{\epsilon}{2}\rbrace$ and $B=\lbrace i\in\mathbb{N}\mid d_2(g^i(y),y_i)\geq\frac{\epsilon}{2}\rbrace$. Then it is clear that $\overline{d}(\lbrace i\in\mathbb{N}\mid d(f^i(x\times y),(x_i\times y_i))\geq \epsilon\rbrace)<\epsilon$. This completes our proof.                      
\end{proof} 

\begin{theorem}
If $f$ has mean ergodic shadowing then for each $k\in\mathbb{N}^+$, $f^k$ has mean ergodic shadowing.  
\label{Pramod5} 
\end{theorem}

\begin{proof}
Let $\epsilon>0$ and $k>1$ be a fixed integer. Let $\delta>0$ be given for $\frac{\epsilon}{k}$ by mean ergodic shadowing of $f$. Let $\lbrace x_0,x_1,...,x_{m_0};x_{m_0+1},x_{m_0+2},...,x_{m_1};x_{m_1+1},...\rbrace$ be a $\delta$-ergodic pseudo orbit for $f^k$, in which $d(f(x_{m_i}),x_{m_i+1})\geq \delta$ for all $i\in\mathbb{N}$ and $d(f(x),y))<\delta$ for any other pair of consecutive points $x$ and $y$ in the sequence. By locating the finite sequence $\lbrace f(x_i),f^2(x_i),...,f^{k-1}(x_i)\rbrace$ between $x_i$ and $x_{i+1}$ for all $i\in\mathbb{N}$, we obtain the $\delta$-ergodic pseudo orbit $\lbrace y_i\rbrace_{i\in\mathbb{N}}=\lbrace x_0,f(x_0),f^2(x_0),...,f^{k-1}(x_0),x_1,f(x_1),f^2(x_1),...,f^{k-1}(x_1),x_2,...\rbrace$ for $f$. 
\medskip

By mean ergodic shadowing of $f$, there is $y\in X$ such that 

lim sup$_{n\to\infty}\frac{1}{n}\Sigma_{i=0}^{n-1}d(f^i(y),y_i)<\frac{\epsilon}{k}$. 

Therefore, 

lim sup$_{n\to\infty}\frac{1}{n}\Sigma_{i=0}^{n-1}d((f^k)^i(y),x_i)$

=lim sup$_{n\to\infty}\frac{1}{n}\Sigma_{i=0}^{n-1}d(f^{ik}(y),y_{ik})$

$\leq$lim sup$_{n\to\infty}\frac{1}{n}\Sigma_{i=0}^{n-1}\Sigma_{j=0}^{k-1}d(f^{ki+j}(y),y_{ki+j})$

$\leq k\times$lim sup$_{n\to\infty}\frac{1}{nk}\Sigma_{l=0}^{nk-1}d(f^l(y),y_l)<\epsilon$. 
\medskip

Thus, we conclude that $f^k$ has mean ergodic shadowing for each $k\in\mathbb{N}^+$.    
\end{proof}    

\begin{Example}    
Let $f:[0,1]\rightarrow [0,1]$ be given by $f(x)=1-x$. Since $f$ is an isometry, it cannot have shadowing. We prove that $f$ cannot have mean ergodic shadowing either. Let $\epsilon<\frac{1}{3}$ and set   
\begin{center}
$a_0=\lbrace 0,1\rbrace$, $a_1=\lbrace 0,1,1,0\rbrace$, $a_2=\lbrace 0,1,0,1,1,0,1,0\rbrace$,..., $a_n=\lbrace \underbrace{0,1,0,1,...,0,1}_\text{(0,1) repeated $n$-times},\underbrace{1,0,...,1,0,1,0}_\text{(1,0) repeated $n$-times}\rbrace$,... 
\end{center}
\medskip

Note that $\lbrace a_0\vee a_1\vee a_2\vee...\vee a_n\vee...\rbrace$ is a $\delta$-ergodic pseudo orbit for any $\delta>0$, where $\vee$ denotes the concatenation operation, for example, $a_0\vee a_1=\lbrace 0,1,0,1,1,0\rbrace$. But it cannot be $\epsilon$-shadowed except on a set of upper density less than $\epsilon$ due to the following reason.   
\medskip

Observe that a tracing point $x$ must come from the $\epsilon$-neighborhood of either zero or one. If it comes from a neighborhood of zero, then the orbit cannot trace the $n$-times repeated $\lbrace 1,0\rbrace$ in the second half of each of $a_n$. If it comes from a neighborhood of one, then the orbit cannot trace the $n$-times repeated $\lbrace 0,1\rbrace$ in the first half of each of $a_n$.  
\end{Example}

\begin{Proposition}
If $f$ has mean ergodic shadowing, then it has $M^{\alpha}$-shadowing for each $\alpha\in (0,1)$.    
\label{2.2.4}
\end{Proposition}

\begin{proof} 
Fix $\alpha\in (0,1)$ and choose $\epsilon_0>0$ such that $\alpha<1-\epsilon_0$. For this $\epsilon_0$ there exists $\delta_0>0$ by mean ergodic shadowing of $f$, such that every $\delta_0$-ergodic pseudo orbit $\xi=\lbrace x_i\rbrace_{i\in\mathbb{N}}$ satisfies $\overline{d}(B^c(x,\xi,\epsilon_0))<\epsilon_0$ and hence, satisfies $\overline{d}(B(x,\xi,\epsilon_0))>1-\epsilon_0>\alpha$ for some $x\in X$. If $\epsilon>\epsilon_0$, then choose the same $\delta_0$ associated with $\epsilon_0$ so that every $\delta_0$-ergodic pseudo orbit $\xi=\lbrace x_i\rbrace_{i\in\mathbb{N}}$ satisfies $\overline{d}(B(x,\xi,\epsilon))\geq\overline{d}(B(x,\xi,\epsilon_0))>1-\epsilon_0>\alpha$. If $\epsilon<\epsilon_0$, then by mean ergodic shadowing there exists $\delta>0$ such that every $\delta$-ergodic pseudo orbit $\xi=\lbrace x_i\rbrace_{i\in\mathbb{N}}$ satisfies $\overline{d}(B(x,\xi,\epsilon))>1-\epsilon>1-\epsilon_0>\alpha$. Since $\alpha$ was chosen arbitrarily, we are through.         
\end{proof} 

\begin{Proposition}
If $f$ has weak asymptotic average shadowing, then it has mean ergodic shadowing.    
\label{2.2.6}  
\end{Proposition}

\begin{proof} 
In view of Theorem 4.3 \cite{WOC}, it is sufficient to show that a $\delta$-ergodic pseudo orbit $\lbrace x_i\rbrace_{i\in\mathbb{N}}$ is a $\delta$-asymptotic average pseudo orbit. 
\medskip

Indeed if $E=\lbrace i\in\mathbb{N}\mid d(f(x_i),x_{i+1})\geq \delta\rbrace$, then $d(E)=0$ and hence, $\overline{d}(E)<\eta$ where $\eta<\frac{\delta}{diam(X)+1}$. Therefore, 
\medskip

lim sup$_{n\to\infty}\frac{1}{n}\Sigma_{i=0}^{n-1}d(f(x_i),x_{i+1})$

$\leq$ lim sup$_{n\to\infty}\frac{1}{n}(diam(X)\#([0,n-1]\cap E)+\eta n)$

$\leq diam(X)\overline{d}(E)+\eta<\delta$. 
\end{proof} 

The following result holds because it is easy to observe that asymptotic average shadowing implies weak asymptotic average shadowing. 

\begin{Corollary}
If $f$ has asymptotic average shadowing, it has mean ergodic shadowing. 
\end{Corollary}

\begin{Lemma}
Let $f$ be chain mixing. If for given $\epsilon>0$, there is $\delta>0$ such that every $\delta$-pseudo orbit is $\epsilon$-shadowed except on a set of upper density less than $\epsilon$, then any $\delta$-ergodic pseudo orbit is $\epsilon$-shadowed except on a set of upper density less than $\epsilon$.    
\label{Pramod1}   
\end{Lemma}

\begin{proof}
The proof follows similarly as in Lemma 4.1 \cite{DH}.  
\end{proof}

\begin{Proposition}
If $f$ has average shadowing, then it has mean ergodic shadowing.
\label{Pramod6} 
\end{Proposition}

\begin{proof}
Since average shadowing implies chain mixing, by Lemma \ref{Pramod1} it is enough to show that for given $\epsilon>0$, there is $\delta>0$ such that every $\delta$-pseudo orbit is $\epsilon$-shadowed except on a set of upper density less than $\epsilon$, by some point in $X$.  
\medskip

Fix $\epsilon>0$ and let $\delta>0$ be given for $\epsilon^2$ by average shadowing of $f$. If $\lbrace x_i\rbrace_{i\in\mathbb{N}}$ is a $\delta$-pseudo orbit, then it is a $\delta$-average pseudo orbit. So, lim sup$_{n\to\infty} \Sigma_{i=0}^{n-1} d(f^i(x),x_i)<\epsilon^2$ for some $x\in X$. 
\medskip

If $E=\lbrace i\in\mathbb{N}\mid d(f^i(x),x_i)\geq \epsilon\rbrace$, then 

$\epsilon^2>$lim sup$_{n\to\infty}\frac{1}{n}\Sigma_{i=0}^{n-1}d(f^i(x),x_i)\geq$ lim sup$_{n\to\infty}\frac{1}{n}(\epsilon\#([0,n-1]\cap E))=\epsilon\overline{d}(E)$. 

Thus $\overline{d}(E)<\epsilon$ and this completes our proof.           
\end{proof} 

The following result holds because it is proved in \cite{GD} that almost average shadowing implies average shadowing.  

\begin{Corollary}
If $f$ has almost average shadowing, then it has mean ergodic shadowing. 
\end{Corollary}

\begin{theorem}
Let $f:X\rightarrow X$ be a continuous map on a compact metric space $X$. If $f$ has shadowing, then the following are equivalent:     
\begin{itemize}
\item [(i)] $f$ has mean ergodic shadowing,  
\item [(ii)] $f$ has average shadowing, 
\item [(iii)] $f$ is totally transitive,  
\end{itemize}
\label{3.13} 
\end{theorem}

\begin{proof}
It follows from Proposition \ref{2.2.4} and Theorem 2.2 \cite{DH} and Theorem \ref{Pramod5} that if $f$ has mean ergodic shadowing, then it is totally chain transitive. Since $f$ has shadowing, $f^k$ has shadowing \cite{AH}. Therefore, to show (i)$\Rightarrow$ (iii) it is enough to show that chain transitivity implies transitivity. Suppose that $f$ is chain transitive. Let $U,V\subset X$ be a pair of non-empty open sets and $x\in U$, $y\in V$. Let $\epsilon>0$ be such that $B(x,\epsilon)\subset U$ and $B(y,\epsilon)\subset V$. Let $\delta>0$ be given for $\epsilon$ by shadowing of $f$. Since $f$ is chain transitive, for any sufficiently large $n>0$ there is a finite set $\lbrace x=x_0,x_1,...,x_n=y\rbrace$ which can be extended to a $\delta$-pseudo orbit $\lbrace x=x_0,x_1,...,x_n=y,f(y),f^2(y)...\rbrace$. By shadowing of $f$ there is $w\in X$ such that $d(w,x)<\epsilon$ and $d(f^n(w),y)<\epsilon$, i.e. $w\in B(x,\epsilon)\subset U$ and $f^n(w)\in B(y,\epsilon)\subset V$ which implies $f^n(U)\cap V\neq\phi$. Thus, $f$ is transitive. The implication (iii)$\Rightarrow$ (ii) is a consequence of Theorem 3.4 \cite{GD}. The implication (ii)$\Rightarrow$ (i) follows from Proposition \ref{Pramod6} without the presence of shadowing.        
\end{proof} 

So, mean ergodic shadowing of a continuous map with shadowing is equivalent to one and hence, all the dynamical notions listed in Theorem 3.8 \cite{KKO} and Theorem 3.4 \cite{GD}.    

\begin{Example}
Let $X=C_1\cup C_2$, where $C_1$ and $C_2$ be disjoint circles in euclidean plane and $f$ be the map sending a point $\theta$ on $C_1$ to $2\theta$ on $C_2$ and vice-versa. Then, $f^2$ is not chain transitive and hence, not totally chain transitive. So by Theorem \ref{3.13}, it cannot have mean ergodic shadowing.    
\end{Example} 

\begin{Example} 
The period doubling map $f:S^1\rightarrow S^1$ given by $f(\theta)=2\theta$ has shadowing and it is totally transitive. So, by Theorem \ref{3.13}, it has mean ergodic shadowing. 
\end{Example} 

\begin{Example}
For any compact metric space $X$, the shift map $\sigma:X^{\mathbb{N}}\rightarrow X^{\mathbb{N}}$ is totally transitive and has shadowing. So, by Theorem \ref{3.13}, it has mean ergodic shadowing. 
\end{Example}

Since finite shadowing is equivalent to shadowing on relatively compact metric spaces, therefore Lemma 3.2 \cite{FG} is true for continuous maps on such spaces. But the following example shows that the converse is not true.         

\begin{Example} 
Let $X=\bigcup_{n\in\mathbb{Z}} (\lbrace n\rbrace\times [0,\frac{1}{2^{\mid n\mid}}])\subset\mathbb{R}^2$ be the subspace of $\mathbb{R}^2$. Let $f$ be the homeomorphism on $X$ given by 
\[f(k,x)=\begin{cases}
(k+1,\frac{1}{2}x) & \textnormal {for $k\geq 0$} 
\\
(k+1,2x) & \textnormal {for $k<0$}  
\end{cases}\]      
Every $\delta$-pseudo orbit $\lbrace x_i\rbrace_{i\in\mathbb{N}}$ can be transformed into a $\delta$-pseudo orbit $\lbrace x_i\rbrace_{i\in\mathbb{Z}}$ by taking $x_i=f^{i}(x_0)$ for all $i\leq -1$. Therefore, using Example 7 \cite{DLRW} we conclude that $f$ has shadowing. We now show that $f$ does not have mean ergodic shadowing which in turn, implies that $f$ does not have ergodic shadowing. 
\medskip

Let $\delta>0$ be given and let $M$ be a subset of $\mathbb{N}$ such that $d(M)=0$. Let $x_i\in \lbrace i\rbrace\times [0,\frac{1}{2^{\mid i\mid}}]$ for $i\in \mathbb{N}\setminus M$ and $x_i\in\lbrace 0\rbrace\times [0,1]$ for $i\in M$. One can find such a sequence with additional property that $d(f(x_i),x_{i+1})<\delta$ for all $i\in\mathbb{N}\setminus M$. Since $M$ has density zero, $\lbrace x_i\rbrace_{i\in\mathbb{N}}$ is a $\delta$-ergodic pseudo orbit. This $\delta$-ergodic pseudo orbit cannot be $\epsilon$-ergodic shadowed because for any $x\in X$, $f^i(x)$ increases along $x$-axis as $i$ goes to $\infty$.  
\end{Example} 

Recall from \cite{WOC} that factors onto a minimal equicontinuous system cannot have $\overline{d}$-shadowing and hence, by Proposition \ref{2.2.4} such a system cannot have mean ergodic shadowing. In particular, factor of a dyadic adding machine cannot have mean ergodic shadowing which is also evident from the following result. 

\begin{theorem}
If $f$ has shadowing and mean ergodic shadowing, it cannot be minimal.     
\end{theorem} 

\begin{proof}
If possible, suppose that $f$ has both shadowing and mean ergodic shadowing. Let $x,y\in X$ with $x\neq y$ and $0<\epsilon<\frac{d(x,y)}{2}$. Let $\delta\in (0,\frac{\epsilon}{2})$ be given for $\frac{\epsilon}{2}$ by shadowing of $f$. Since $f$ has shadowing and mean ergodic shadowing, by Proposition \ref{3.13} it is totally transitive. So, there is $k\geq 1$ such that $f^k(B(x,\delta))\cap B(x,\delta)\neq \phi$, i.e. there is $z\in B(x,\delta)$ such that $f^k(z)\in B(x,\delta)$. Then, $\lbrace z,f(z),f^2(z),...,f^{k-1}(z),z,f(z),f^2(z),...f^{k-1}(z),...\rbrace$ is a $\delta$-pseudo orbit for $f$ and hence, there is $w\in X$ such that $d(f^{kn}(w),z)<\frac{\epsilon}{2}$ for all $n\in\mathbb{N}$. By Zorn's lemma there is a syndetically recurrent point $p\in\overline{O_{f^k}(w)}$ for $f^k$. Observe that $d(f^{kn}(p),z)\leq\frac{\epsilon}{2}$ for all $n\in\mathbb{N}$ and therefore, $d(x,f^{kn}(p))\leq d(x,z)+d(z,f^{kn}(p))<\frac{\epsilon}{2}+\frac{\epsilon}{2}=\epsilon$. Thus, $f^{kn}(p)\in B(x,\epsilon)$ and since $d(x,y)>2\epsilon$, $y\notin \overline{O_{f^k}(p)}$. Thus, $f^k$ is not minimal. As $f^k$ is transitive, there is $q\in X$ such that $\overline{O_{f^k}(q)}=X$. Since $f^k$ is not minimal, $q\notin M(f^k)$, and hence, $q\notin M(f)$. Since every point of a minimal system is syndetically recurrent, $f$ cannot be minimal.     
\end{proof}

Then by Lemma 3.2 \cite{FG}, a map with ergodic shadowing cannot be minimal. In contrast, the following result shows that syndetically recurrent points are present everywhere in the phase space of such a system.     

\begin{theorem}
If $f$ has ergodic shadowing, then $\overline{M(f)}=X$. 
\end{theorem}   

\begin{proof}
It is similar to the proof given in Theorem 3.2 \cite{GD} using Lemma 3.1, 3.2 of \cite{FG}.  
\end{proof}

Then Proposition \ref{2.2.4} and Theorem 7.2 \cite{WOC} imply that a map with ergodic shadowing is totally syndetically transitive. 

\begin{theorem}
A minimal system with mean ergodic shadowing is weakly mixing. 
\end{theorem} 

\begin{proof}
It follows from Theorem 7.10 \cite{WOC} and Proposition \ref{2.2.4}.    
\end{proof} 

\section{Mean Ergodic Shadowing and Orbital Convergence}    

In this section, we study the behaviour of mean ergodic shadowing under orbital convergence of sequence of maps with mean ergodic shadowing.   
\medskip

Define $MESh(f,x,\epsilon)=\lbrace \delta>0|$ every $\delta$-ergodic pseudo orbit for $f$ through $x$ is $\epsilon$-shadowed in average by some point in $X\rbrace$.  
\medskip

The following notion of convergence is stronger than uniform convergence. Therefore, the limit function $f$ turns out to be continuous. 

\begin{Definition}
Let $X$ be a compact metric space and for each $n\in\mathbb{N}$, $f_n:X\rightarrow X$ be continuous maps. Then $\lbrace f_n\rbrace_{n\in\mathbb{N}}$ is said to be orbitally convergent to $f:X\rightarrow X$ if for every $\epsilon>0$ there exists a natural number $N$ such that $d(f_n^k(x),f^k(x))<\epsilon$ for all $n\geq N$, all $x\in X$ and all $k\in\mathbb{N}$.  
\end{Definition}   

\begin{theorem}
Let $\lbrace f_n\rbrace_{n\in\mathbb{N}}$ be a sequence of continuous maps on a compact metric space $X$. If $\lbrace f_n\rbrace$ is orbitally converging to $f$, then $f$ has mean ergodic shadowing if and only if $\bigcup_{m\geq 1} \bigcap_{n\geq m} MESh(f_n,x,\epsilon)\neq \phi$ for all $x\in X$ and for every $\epsilon>0$.      
\end{theorem}  

\begin{proof}
Suppose that $f$ has the mean ergodic shadowing. Then, we want to show that $\bigcup_{m\geq 1} \bigcap_{n\geq m} MESh(f_n,x,\epsilon)\neq \phi$ for all $x\in X$. Let us fix $x\in X$ and $\epsilon>0$. Further, let $0<\delta<\frac{\epsilon}{2}$ be given for $\frac{\epsilon}{2}$ by the mean ergodic shadowing of $f$. So, every $\delta$-ergodic pseudo orbit through $x$ is $\frac{\epsilon}{2}$-shadowed in average by some point in $X$. Since $\lbrace f_n\rbrace$ is orbitally converging to $f$, we choose $M\geq 1$ such that $d(f_m^k(z),f^k(z))<\frac{\delta}{2}$ for each $m\geq M$, for all $z\in X$ and for each $k\in \mathbb{N}^+$. For a fixed $m\geq M$, let us consider a $\frac{\delta}{2}$-ergodic pseudo orbit $\gamma=\lbrace x_i\rbrace_{i\in\mathbb{N}}$ for $f_m$ through $x$. Let $N\subset \mathbb{N}$ with $d(N)=0$ be such that $d(f_m(x_i),x_{i+1})<\frac{\delta}{2}$ for all $i\in \mathbb{N}\setminus N$. Then $d(f(x_i),x_{i+1})\leq d(f(x_i),f_m(x_i))+d(f_m(x_i),x_{i+1})<\frac{\delta}{2}+\frac{\delta}{2}=\delta$ for all $i\in \mathbb{N}\setminus N$. Thus, $\gamma$ is a $\delta$-ergodic pseudo orbit for $f$ through $x$. So, there exists $y\in X$ such that lim sup$_{n\to\infty}\frac{1}{n}\Sigma_{i=0}^{n-1}d(f^i(y),x_i)<\frac{\epsilon}{2}$. Therefore, 
\medskip

lim sup$_{n\to \infty} \frac{1}{n} \Sigma_{i=0}^{n-1} d(f_m^i(y),x_i)$ 

$\leq$ lim sup$_{n\to\infty} \frac{1}{n}\Sigma_{i=0}^{n-1} d(f_m^i(y),f^i(y))+$ lim sup$_{n\to\infty} \Sigma_{i=0}^{n-1} d(f^i(y),x_i)$  

$<\frac{\epsilon}{2}+\frac{\epsilon}{2}=\epsilon$. 
\medskip

Thus, every $\frac{\delta}{2}$-ergodic pseudo orbit for $f_m$ through $x$ is $\epsilon$-shadowed in average by some point in $X$. Since $m\geq M$ was arbitrary, we conclude that $\frac{\delta}{2}\in \bigcap_{n\geq m} MESh(f_n,x,\epsilon)$. Since $x$ and $\epsilon$ was arbitrary, we have $\bigcup_{m\geq 1} \bigcap_{n\geq m} MESh(f_n,x,\epsilon)\neq \phi$ for all $x\in X$ and for every $\epsilon>0$. This proves the direct implication. 
\medskip

To prove the reverse implication, let us assume that $\bigcup_{m\geq 1} \bigcap_{n\geq m} MESh(f_n,x,\frac{\epsilon}{2})\neq \phi$ for given $\epsilon>0$ and given $x\in X$. We choose $M$ such that $\bigcap_{n\geq M} MESh(f_n,x,\frac{\epsilon}{2})\neq \phi$. So, there exists $0<\delta_x<\epsilon$ such that for each $m\geq M$, every $\delta_x$-ergodic pseudo orbit for $f_m$ through $x$ can be $\frac{\epsilon}{2}$-shadowed in average by some point in $X$. Since $\lbrace f_n\rbrace$ converges orbitally to $f$, we can choose $N\geq M$ such that $d(f_m^k(z),f^k(z))<\frac{\delta_x}{2}$ for each $m\geq N$, for all $z\in X$ and for each $k\in \mathbb{N}^+$. Let $\gamma=\lbrace x_i\rbrace_{i\in\mathbb{N}}$ be a $\frac{\delta_x}{2}$-ergodic pseudo orbit for $f$ through $x$, i.e., there exists $K\subset \mathbb{N}$ with $d(K)=0$ such that $d(f(x_i),x_{i+1})<\frac{\delta_x}{2}$ for all $i\in\mathbb{N}\setminus K$. Then $d(f_N(x_i),x_{i+1})\leq d(f_N(x_i),f(x_i))+d(f(x_i),x_{i+1})<\frac{\delta_x}{2}+\frac{\delta_x}{2}=\delta_x$ for all $i\in \mathbb{N}\setminus K$. Thus, $\gamma$ is a $\delta_x$-ergodic pseudo orbit for $f_N$ through $x$ and hence, can be $\frac{\epsilon}{2}$-shadowed in average by some point in $X$. So, there exists $y\in X$ such that lim sup$_{n\to\infty}\frac{1}{n}\Sigma_{i=0}^{n-1}d(f_N^i(y),x_i)<\frac{\epsilon}{2}$. Therefore,  
\medskip

lim sup$_{n\to \infty} \frac{1}{n} \Sigma_{i=0}^{n-1} d(f^i(y),x_i)$ 

$\leq$ lim sup$_{n\to\infty} \frac{1}{n}\Sigma_{i=0}^{n-1} d(f^i(y),f_N^i(z))+$ lim sup$_{n\to\infty} \Sigma_{i=0}^{n-1} d(f_N^i(y),x_i)$ 

$<\frac{\epsilon}{2}+\frac{\epsilon}{2}=\epsilon$.  
\medskip

Thus, for every $\epsilon>0$ there exists $\delta_x>0$ such that every $\delta_x$-ergodic pseudo orbit through $x$ is $\epsilon$-shadowed in average by some point in $X$. This would imply that for every $\epsilon>0$ there exists $\delta_x>0$ such that every $\delta_x$-ergodic pseudo orbit through $B[x,\delta_x]$ is $\epsilon$-shadowed in average by some point in $X$. If not, suppose there exists $\epsilon>0$ for which there does not exist any $\delta>0$ such that every $\delta$-ergodic pseudo orbit is $\epsilon$-shadowed in average by some point in $X$. Thus for each $k\geq 1$, we can choose $\frac{1}{k}$-ergodic pseudo orbit $\eta^k=\lbrace \eta_n^k\rbrace_{n\in\mathbb{N}}$ with $\eta_0^k\in B[x,\frac{1}{k}]$ such that $\eta^k$ cannot be $\epsilon$-shadowed in average by some point in $X$. By compactness of $X$ we choose $K$ large enough so that $d(f(\eta_0^K),f(x))\leq \frac{\delta_x}{2}$ and $\frac{1}{K}\leq \frac{\delta_x}{2}$. Let $N\subset \mathbb{N}$ with $\overline{d}(N)=0$ such that $d(f(\eta_i^K),\eta_{i+1}^K)\leq \frac{1}{K}$ for all $i\in\mathbb{N}\setminus N$.    

Define $Y=\lbrace y_n\rbrace_{n\in\mathbb{N}}$ by \[y_n=\begin{cases}
x & \textnormal {if $n=0$}  
\\
\eta_n^K & \textnormal {if $n\geq 1$}      
\end{cases}\] 

For $i\in \mathbb{N}^+\setminus N$, $d(f(y_i),y_{i+1})=d(f(\eta_i^K),\eta_{i+1}^K)\leq \frac{1}{K}\leq \frac{\delta_x}{2}<\delta_x$.  
\medskip

Thus, $Y=\lbrace y_n\rbrace_{n\in\mathbb{N}}$ is a $\delta_x$-ergodic pseudo orbit for $f$ through $x$. So, we can choose $z\in X$ such that lim sup$_{n\to\infty} \frac{1}{n}\Sigma_{i=0}^{n-1} d(f^i(z),y_i)<\epsilon$. Then,  
\medskip

lim sup$_{n\to\infty} \frac{1}{n}\Sigma_{i=0}^{n-1} d(f^i(z),\eta_i^K)$ 

$=$lim sup$_{n\to\infty} \frac{1}{n} d(z,\eta_0^K)+$lim sup$_{n\to\infty} \frac{1}{n} \Sigma_{i=1}^{n-1} d(f^i(z),\eta_i^K)<0+\epsilon=\epsilon$.    
\medskip

It follows that $\lbrace \eta_i^K\rbrace_{i\in\mathbb{N}}$ is $\epsilon$-shadowed in average by some point in $X$. This contradiction makes the assumption false.  
\medskip

Now, for each $x\in X$ and every $\epsilon>0$, there exists $\delta_x>0$ such that every $\delta_x$-ergodic pseudo orbit through $B[x,\delta_x]$ is $\epsilon$-shadowed in average by some point in $X$. By compactness of $X$, we have $X=\bigcup_{i=1}^m B[x_i,\delta_i]$, where $\delta_i=\delta_{x_i}$ for $1\leq i\leq m$. Choose $\delta=$min$_{1\leq i\leq m} \lbrace \delta_i\rbrace$ and let $\lbrace y_n\rbrace_{n\in\mathbb{N}}$ be a $\delta$-ergodic pseudo orbit for $f$. Clearly, $y_0\in B[x_i,\delta_i]$ for some $1\leq i\leq n$. Thus, $\lbrace y_n\rbrace_{n\in\mathbb{N}}$ is a $\delta_i$-ergodic pseudo orbit for $f$ through $B[x_i,\delta_i]$. So, one can find $y\in X$ such that lim sup$_{n\to\infty} \frac{1}{n} \Sigma_{i=0}^{n-1} d(f^i(y),y_i)<\epsilon$. This completes a proof of the fact that $f$ has the mean ergodic shadowing property.   
\end{proof} 

\textit{Acknowledgement:} {\footnotesize Part of this research was carried out while the first author was pursuing PhD in the Department of Mathematics, University of Delhi, India, supported by the Department of Science \& Technology, Government of India, under INSPIRE Fellowship (Registration No.-IF150210) Program}.

\end{document}